\documentclass{amsart}
\pdfoutput=1
\usepackage{fullpage, amsmath, amsthm, amsfonts, amssymb, amscd}
\usepackage{mathrsfs}
\usepackage{hyperref}
\usepackage{stmaryrd} 

\newtheorem{thm}{Theorem}[section]
\newtheorem*{thm*}{Theorem}
\newtheorem*{thmA}{Theorem A}
\newtheorem*{thmB}{Theorem B}
\newtheorem{coro}[thm]{Corollary}
\newtheorem*{coro*}{Corollary}
\newtheorem{prop}[thm]{Proposition}
\newtheorem{lemma}[thm]{Lemma}

\newtheorem*{claim}{Claim}

\theoremstyle{remark}

\theoremstyle{definition}
\newtheorem{remark}[thm]{Remark}

\newtheorem{defn}[thm]{Definition}
\newtheorem{exa}[thm]{Example}

\newtheorem*{ackn}{Acknowledgements}

\newenvironment{pfClaim}{\,{\em Proof of claim:}}{\hfill$\boxtimes$\newline\newline}

\title{A Note on Injectivity of Frobenius on Local Cohomology \\ of Global Complete Intersections}
\author{Eric Canton}
\address{Department of Mathematics, University of Nebraska -- Lincoln, 203 Avery Hall, Lincoln NE 68588}
\email{ecanton2@math.unl.edu}

\thanks{The author was partially supported by NSF grant DMS \#1247354.}

\begin{document}
\begin{abstract}
  Given a graded complete intersection ideal $J = (f_1, \dots, f_c) \subseteq k[x_0, \dots, x_n] = S$, where $k$ is a field of characteristic 
  $p > 0$ such that $[k:k^p] < \infty$, we show that if $S/J$ has an isolated non-F-pure point then the Frobenius action on top local cohomology 
  $H^{n+1-c}_\mathfrak{m}(S/J)$ is injective in sufficiently negative degrees, and we compute the least degree of any kernel element.
  If $S/J$ has an isolated singularity, we are also able to give an effective bound on $p$ ensuring the Frobenius action on 
  $H^{n+1-c}_\mathfrak{m}(S/J)$ is injective in all negative degrees, extending a result of Bhatt and Singh in the hypersurface case. 
\end{abstract}
\maketitle

\section{Introduction}
Let $k$ be a field of characteristic $p > 0$ such that $[k:k^p] < \infty$. 
We study the kernel of the Frobenius action on local cohomology of a graded complete intersection 
$R = k[x_0, \dots, x_n]/(f_1, \dots, f_c)$ under the assumption that Frobenius is 
pure on $R_\mathfrak{p}$ for all primes $\mathfrak{p}$ different from the graded maximal ideal $\mathfrak{m} = (x_0, \dots, x_n)$. 
Let $J = (f_1, \dots, f_c)$. Recall that $R_\mathfrak{p}$ has a pure Frobenius homomorphism for a prime $\mathfrak{p} \ne \mathfrak{m}$ of $S$ containing $J$ if and only if 
$(f_1\cdots f_c)^{p-1} \not\in \mathfrak{p}^{[p]}$ \cite{Fedder}, where $\mathfrak{p}^{[p]} = (a^p \,|\, a \in \mathfrak{p})$. The first of our main results is the following.

\begin{thmA}[\ref{Injectivity Theorem}]
  Suppose $(f_1 \cdots f_c)^{p-1} \not\in \mathfrak{p}^{[p]}$ for any prime $\mathfrak{p} \ne \mathfrak{m}$ with $J \subseteq \mathfrak{p}$, but $(f_1 \cdots f_c)^{p-1} \in \mathfrak{m}^{[p]}$.
  Set $a(R) = - (n+1) + \sum_1^c \deg(f_i)$. Let $\tau$ be the smallest ideal of $S$ such that $J \subseteq \tau$ and
  $(f_1\cdots f_c)^{p-1} \in \tau^{[p]}$. Note in this case $\sqrt{\tau} = \mathfrak{m}$; set $\ell = \max\{s \,|\, \mathfrak{m}^s \not\subseteq \tau\} < \infty$. Then 
  the below Frobenius action is injective:
  \[ F: H^{n+1-c}_\mathfrak{m}(R)_{< a(R) - \ell}  \to H^{n+1-c}_\mathfrak{m}(R)_{< p(a(R) - \ell)}. \]
\end{thmA}

The proof also shows that this bound is sharp: there always exists a class $\alpha \in H^{n+1-c}_\mathfrak{m}(R)$ of degree 
$a(R) - \ell$ such that $F(\alpha) = 0$. The above theorem now yields the following characterization of 
graded complete intersection quotients such that $\sqrt{\tau} = \mathfrak{m}$.
The main advantage of this characterization is that it depends only on $n$, the codimension
$c$, and the degree $d := \sum_1^c \deg(f_j)$. 

\begin{coro*}[\ref{Degree Bound}]
  Let $\tau$ be as in Theorem A and suppose $\tau \ne S$. The condition $\sqrt{\tau} = \mathfrak{m}$ is equivalent to the Frobenius action on 
  $H^{n+1-c}_\mathfrak{m}(R)$ being injective in degrees $< -(n+1-c)d$. 
\end{coro*}

An easy way to ensure that a complete intersection $R$ satisfies the hypothesis of theorem A is to 
assume $R$ has isolated singularity. Under this additional assumption we are able 
to combine an argument of Fedder \cite[Theorem 2.1]{FedderGradedCI} with one of Bhatt and Singh to generalize 
\cite[Theorem 3.5]{Bhatt-Singh}.

\begin{thmB}[\ref{Bhatt-Singh style}]
  Assume $R$ has an isolated singularity. If $p \ge (n+1-c)(d-c)$ then the Frobenius action on $H^{n+1-c}_\mathfrak{m}(R)$ is 
  injective in negative degrees. 
\end{thmB}

\begin{ackn}
  I would like to thank my advisor, Wenliang Zhang, for numerous helpful discussions. I would also like to thank Tom Marley, Anurag Singh, 
  and Thanh Vu for useful conversations, and the referee for many suggestions which improved the readability of this paper. 
\end{ackn}

\section{Notation and Conventions}
The letter $q$ will always denote an integer power $p^e$ of a prime $p > 0$. Let $S$ be a $\mathbb{Z}$-graded Noetherian ring containing a field $k$ of 
characteristic $p > 0$ with $[k:k^p] < \infty$ and posessing a unique maximal homogeneous ideal $\mathfrak{m}$. For any ideal $I$ of $S$, define 
$I^{[p]} = (f^p \,|\, f \in I)$. Similarly, we define $I^{[q]} = (f^q \,|\, f\in I)$ for all $q \ge 0$. 

\begin{defn}[Regularity]\label{regularity}
  Suppose $M$ is a $\mathbb{Z}$-graded $S$-module such that $M_\ell = 0$ for $\ell \gg 0$. 
  Define the {\em Castelnuovo-Mumford regularity} of $M$ to be
  \[ \mathrm{reg}(M) = \max\{\ell \,\,|\,\, (M)_\ell \ne 0\}. \]
\end{defn}

\begin{exa}\label{regularity example}
  If $I \subseteq S$ is a homogeneous ideal such that $\sqrt{I} = \mathfrak{m}$ then for all $\ell \gg 0$ we know $\mathfrak{m}^\ell \subseteq I$. Thus $(S/I)_\ell = 0$ for
  $\ell \gg 0$ and we can define $\mathrm{reg}(S/I)$ as above. The regularity of $S/I$ is the unique integer $a$ such that $\mathfrak{m}^a \not\subseteq I$ but
  $\mathfrak{m}^{a+1} \subseteq I$. 
\end{exa}

We introduce the following 
\begin{defn}
  Let $0 \ne I \subsetneq S$ be a proper homogeneous ideal. For each $q \ge 0$ define
  \[ M_q(I) = \max\left\{\ell \,\,\left|\,\, (\mathfrak{m}^{[q]} : I) \subseteq \mathfrak{m}^{[q]} + \mathfrak{m}^\ell \right.\right\}. \]
\end{defn}

\begin{exa}
  Suppose $S = k[x, y]$ and let $I = (x^a, y^b)$ with $a \le b$. For $q \le \min\{a, b\}$, $(\mathfrak{m}^{[q]} : I) = S$ so 
  $M_q(I) = 0$. For $a < q \le b$ the colon $(\mathfrak{m}^{[q]} : I)$ becomes $\mathfrak{m}^{[q]} + (x^{q - a})$ so $M_q(I) = q - a$.
  Finally for $q > b$ the colon is $\mathfrak{m}^{[q]} + (x^{q-a}y^{q-b})$ and so $M_q(I) = 2q - (a + b)$ for all $q \ge b$. In particular, 
  $2q - M_q(I) = a + b$ for all large $q$. Note that since $\mathfrak{m}^{a+b-1} \subseteq I$ but $\mathfrak{m}^{a+b-2}\not\subseteq I$, $2q - M_q(I) = \mathrm{reg}(S/I) + 2$. 
  Our next proposition shows this is always the case for $\mathfrak{m}$-primary $I$.
\end{exa}

\begin{lemma}\label{colon containment}
  Let $I$ be any ideal of $S$. We have a containment $(\mathfrak{m}^{[q]} : I) \subseteq (\mathfrak{m}^{[q]} : \mathfrak{m}^\ell)$ for all $q \gg 0$ if and only if $\mathfrak{m}^\ell \subseteq I$.
\end{lemma}
\begin{proof}
  Very generally, if $J$ and $L$ are ideals in a zero dimensional Gorenstein ring then we have a containment of colon ideals $(0: J) \subseteq (0 : L)$
  if and only if $L \subseteq J$. Since $S/\mathfrak{m}^{[q]}$ is zero dimensional and Gorenstein for all $q \ge 0$ it follows that for a specific
  $q$ we have $(\mathfrak{m}^{[q]} : I) \subseteq (\mathfrak{m}^{[q]} : \mathfrak{m}^\ell)$ if and only if $\mathfrak{m}^\ell + \mathfrak{m}^{[q]} \subseteq I + \mathfrak{m}^{[q]}$. For large $q$ and any fixed $\ell$
  we know $\mathfrak{m}^{[q]} \subseteq \mathfrak{m}^\ell$. Therefore, $\mathfrak{m}^\ell \subseteq \bigcap_{q \gg 0} (I + \mathfrak{m}^{[q]})$ if and only if 
  $(\mathfrak{m}^{[q]} : I) \subseteq (\mathfrak{m}^{[q]} : \mathfrak{m}^\ell)$ for all $q \gg 0$. Krull's intersection theorem implies $I = \bigcap_{q \gg 0} (I + \mathfrak{m}^{[q]})$,
  and the statement of the lemma follows.
\end{proof}

\begin{prop}\label{Regularity Proposition}
  Let $S = k[x_0, \dots, x_n]$ and $I \subsetneq S$ a proper ideal.
  \begin{enumerate}
    \item $(n+1)q - M_q(I) \le n + \ell$ for $q \gg 0$ if and only if $\mathfrak{m}^\ell \subseteq I$. \label{first part}
    \item If $\sqrt{I} = \mathfrak{m}$ then for $q \gg 0$, 
      \[ (n+1)q - M_q(I) = \mathrm{reg}(S/I) + (n+1).\]
  \end{enumerate}
\end{prop}
\begin{proof}
  To establish the first statement, suppose $(n+1)q - M_q(I) \le n + \ell$ for all $q \ge q_0$. Then by definition 
  \[ (\mathfrak{m}^{[q]} : I) \subseteq \mathfrak{m}^{[q]} + \mathfrak{m}^{(n+1)q - (n + \ell)}. \]
  Bhatt and Singh show \cite[3.2]{Bhatt-Singh} that the ideal on the right is $(\mathfrak{m}^{[q]} : \mathfrak{m}^\ell)$. Since this containment
  holds for all $q \ge q_0$ we conclude that $\mathfrak{m}^\ell \subseteq I$ using \eqref{colon containment}. This argument is easily reversed, and the 
  statement follows.

  If $\sqrt{I} = \mathfrak{m}$ then statement \eqref{first part} shows that for $q \gg 0$, $(n+1)q - M_q(I) = n + \ell$ where $\ell$ is the least
  integer such that $\mathfrak{m}^\ell \subseteq I$. Using \eqref{regularity example} we see that $\ell = \mathrm{reg}(S/I) + 1$. 
\end{proof}

\begin{remark}\label{lch identification}
  Let $S = k[x_0, \dots, x_n]$ and let $f_1, \dots, f_c \in S$ be homogeneous forms which form a regular sequence. Set $J = (f_1, \dots, f_c)$,
  $R = S/J$, $f = \prod_1^c f_j$, and $d = \deg(f) = \sum_1^c \deg(f_j)$. Then the Koszul complex $K_\bullet$ on the 
  forms $f_j$ gives the minimal graded free resolution of $R$ over $S$, and after applying $H_\mathfrak{m}^*$ to $K_\bullet$ we may identify
  \[\tag{$\star$}\label{identify} 
    H_\mathfrak{m}^{n+1-c}(R) = \mathrm{Ann}_{H_\mathfrak{m}^{n+1}(S)[-d]}(J).
  \]
  If we compute $H_\mathfrak{m}^{n+1}(S)$ via the \v{C}ech complex \v{C}$(x_0, \dots, x_n)$,
  then we can represent classes $\alpha \in H_\mathfrak{m}^{n+1-c}(R)$ by $[g/x^q]$ where $g \in S$, $x^q := (x_0\cdots x_n)^q$, and
  this class is zero if and only if $g \in \mathfrak{m}^{[q]}$. The identification \eqref{identify} is degree-preserving, so that
  classes $\alpha \in H_\mathfrak{m}^{n+1-c}(R)_t \subseteq H_\mathfrak{m}^{n+1}(S)[-d]_t = H_\mathfrak{m}^{n+1}(S)_{t - d}$ are represented by 
  $[g/x^q]$ with $g$ homogeneous, $g \in (\mathfrak{m}^{[q]} : J)$, and $\deg(g) - (n+1)q = t - d$. 
\end{remark}

\begin{remark}\label{frobenius action}
  Using notation from \eqref{lch identification}, the Frobenius homomorphism on $R$ lifts to a chain map
  $F_\bullet: K_\bullet \to K_\bullet$ which is {\em not} $S$-linear, though it is $\mathbb{Z}$-linear. In the case $c = 1$ (so that $f_1 = f$) 
  the Koszul complex and chain map associated to the Frobenius homomorphism take the form
  \[\begin{CD}
      0 @>>> S[-d] @>f>> S @>>> R @>>> 0 \\
      @.    @Vf^{p-1}F VV @VF VV @VF VV \\
      0 @>>> S[-d] @>f>> S @>>> R @>>> 0.
  \end{CD}\]
  For any $c$ the map $F_c: S[-d] \to S[-d]$ is still given by $f^{p-1}F = (f_1\cdots f_c)^{p-1}F$, so again identifying $H_\mathfrak{m}^{n+1-c}(R)$
  as in \eqref{identify} we can describe the Frobenius action on local cohomology of $R$ as $[g/x^q] \mapsto [f^{p-1}g^p/x^{pq}]$. 
\end{remark}

\section{Injectivity of Frobenius in codimension $c$}\label{Theorem section}

\begin{remark}
  Let $I \subsetneq k[x_0, \dots, x_n]$ be a proper homogeneous ideal and denote $k[x_0, \dots, x_n]/I$ by $A$. It is well-known that the local
  cohomology modules $H_{(x_0, \dots, x_n)}^i(A)$ are graded Artinian modules. We define the {\em a-invariant} of $A$, denoted $a(A)$, as
  $\mathrm{reg}\left(H_{(x_0, \dots, x_n)}^{\dim(A)}(A)\right)$. In the case $A = k[x_0, \dots, x_n]$ it is straightforward to see 
  ({\em e.g.}, using graded local duality) that  $a(A) = -(n+1)$, and if $I = (f_1, \dots, f_c)$ is generated by a homogeneous regular 
  sequence  then $a(A) = -(n+1) + \sum_1^c \deg(f_j)$. 
\end{remark}

\begin{thm}\label{Injectivity Theorem}
  Let $J = (f_1, \dots, f_c) \subseteq k[x_0, \dots, x_n] = S$ be an ideal generated by a homogeneous regular sequence and set $R = S/J$.
  Define $\tau$ as the smallest ideal of $S$ such that $J \subseteq \tau$ and $(f_1\cdots f_c)^{p-1} \in \tau^{[p]}$. Assume 
  $\sqrt{\tau} = \mathfrak{m} := (x_0, \dots, x_n)$. Then the below Frobenius action is injective:
  \[ F: H_\mathfrak{m}^{n+1-c}(R)_{< a(R) -\mathrm{reg}(S/\tau) } \to H_\mathfrak{m}^{n+1-c}(R)_{< p(a(R) -\mathrm{reg}(S/\tau)) }. \]
\end{thm}
\begin{proof}
  Let $f = f_1\cdots f_c$ and $d = \deg(f)$. Note that $a(R) = d - (n+1)$.
  We use \eqref{lch identification} and \eqref{frobenius action} to identify $H_\mathfrak{m}^{n+1-c}(R)$ with $\mathrm{Ann}_{H_\mathfrak{m}^{n+1}(S)[-d]}(J) =: T$ and 
  the Frobenius action on $H_\mathfrak{m}^{n+1-c}(R)$ with $f^{p-1}F|_T$. Aiming for a contradiction, suppose there exists nonzero $\alpha \in T$ 
  such that $f^{p-1}F(\alpha) = 0$ but $\deg(\alpha) < (d - (n+1) -\mathrm{reg}(S/\tau)) - d = -\mathrm{reg}(S/\tau) - (n+1)$. 
  We may represent $\alpha$ by $[g/x^{q}]$ for $q \gg 0$ and $g \in (\mathfrak{m}^{[q]} : J)\setminus \mathfrak{m}^{[q]}$. Using this representation,
  \begin{align*}
    f^{p-1}F(\alpha) = 0 &\iff f^{p-1}g^p \in \mathfrak{m}^{[pq]}\\
                         &\iff f^{p-1} \in (\mathfrak{m}^{[pq]}:g^p) = (\mathfrak{m}^{[q]}:g)^{[p]}.
  \end{align*}
  Since $[g/x^q] \in T$ we know $J \subseteq (\mathfrak{m}^{[q]} : g)$, so that $f^{p-1} \in (\mathfrak{m}^{[q]} : g)^{[p]}$ is equivalent to 
  $\tau \subseteq (\mathfrak{m}^{[q]} : g)$. This in turn is equivalent to $g \in (\mathfrak{m}^{[q]} : \tau)$. Since $g \not\in \mathfrak{m}^{[q]}$,
  $\deg(g) \ge M_q(\tau)$. Using \eqref{Regularity Proposition} we conclude
  \[ -\mathrm{reg}(S/\tau) - (n+1) = M_q(\tau) - (n+1)q  \le \deg(g) - (n+1)q = \deg(\alpha) < -\mathrm{reg}(S/\tau) - (n+1), \]
  a contradiction. 
\end{proof}

\begin{remark}\label{sharpness}
  Theorem \ref{Injectivity Theorem} is sharp: if we take $g \in (\mathfrak{m}^{[q]} : \tau) \setminus \mathfrak{m}^{[q]}$ of degree $M_q(\tau)$ for $q \gg 0$ then 
  $[g/x^q] \mapsto [f^{p-1}g^p/x^{pq}] = 0$. 
\end{remark}

\begin{coro}\label{Degree Bound}
  Using notation from \eqref{Injectivity Theorem}, assume $\tau \ne S$. Then $\sqrt{\tau} = \mathfrak{m}$ if and only if the below Frobenius action 
  is injective:
  \[ F: H_\mathfrak{m}^{n+1-c}(R)_{< -(n+1-c)d} \to H_\mathfrak{m}^{n+1-c}(R)_{< -p(n+1-c)d}. \]
\end{coro}
\begin{proof}
  If $\sqrt{\tau} \ne \mathfrak{m}$ then the sequence $\{M_q(\tau) - (n+1)q\}_{q \ge 1}$ is unbounded below by 
  \eqref{Regularity Proposition}. Remark \ref{sharpness} shows that there always exists a kernel element of
  degree $M_q(\tau) - (n+1)q$ for any $q$. 

  Suppose $\sqrt{\tau} = \mathfrak{m}$. Then by \eqref{Injectivity Theorem} it suffices to show 
  \[ -(n+1-c)d \le d - (n+1) -\mathrm{reg}(S/\tau). \]
  Towards this end, since $\sqrt{\tau} = \mathfrak{m}$ we know that there must exist generators $\phi_1, \dots, \phi_{n+1-c} \in \tau$ such that
  $\underline{\phi}, \underline{f}$ is a regular sequence on $S$. Let $\mathfrak{b} = (\underline{\phi}, \underline{f})$. By \cite[2.4]{BMS} we may choose the $\phi_i$ so that
  $\deg{\phi_i} \le p^{-1} d(p-1) < d$. Thus the Hilbert series $HS(S/\mathfrak{b}, t)$ is
  \[ \frac{\Pi_{i=1}^{n+1-c}(1 - t^{\deg{\phi_i}}) \Pi_{j=1}^c (1-t^{d_i})}{(1 - t)^{n+1}} \]
  which is a polynomial of degree 
  \[ \left(\sum_1^{n+1-c} \deg(\phi_i) \right) + \left(\sum_1^c \deg{f_j}\right) - (n+1) < (n+1-c)d + d - (n+1). \]
  Therefore, $\mathfrak{m}^{(n+1-c)d + d - (n+1)+1}\subseteq \mathfrak{b} \subseteq \tau$ and we conclude
  \begin{align*}
    -\mathrm{reg}(S/\tau) - (n+1) + d &\ge -((n+1-c)d + d - (n+1)) - (n+1) + d\\
                            &= -(n+1-c)d. 
  \end{align*}                                                                             
\end{proof}

We conclude this section by extending Bhatt and Singh's result \cite[3.5]{Bhatt-Singh} using a generalization of their method with an 
application of the determinant trick used in the proof of \cite[2.1]{FedderGradedCI}. 

\begin{thm}\label{Bhatt-Singh style}
  Using the notation from \eqref{Injectivity Theorem}, let $\mathrm{Jac}(R)$ be the ideal of $(c\times c)$ minors of the Jacobian matrix 
  $(\partial f_j/\partial x_i)$, $0 \le i \le n$, $1 \le j \le c$. Suppose $\sqrt{\mathrm{Jac}(R) + J} = \mathfrak{m}$. If $p \ge (n + 1 - c)(d-c)$ then the below 
  Frobenius action is injective:
  \[ F: H^{n+1-c}_\mathfrak{m}(R)_{<0} \to H^{n+1-c}_\mathfrak{m}(R)_{<0}. \]
\end{thm}
\begin{proof}
  Again using the identification $H_\mathfrak{m}^{n+1-c}(R) = \mathrm{Ann}_{H_\mathfrak{m}^{n+1}(S)[-d]}(J) =: T$ of \eqref{lch identification}, suppose there exists 
  $0 \ne \alpha = [g/x^{q/p}] \in T_{\le -1}$ with $f^{p-1}F(\alpha) = [f^{p-1}g^p/x^q] = 0$; then $\deg(g) - (n+1)(q/p) \le -d - 1$. 
  We show this implies $p < (n + 1 - c)(d - c)$, contradicting our assumption on $p$.

  Define $(t_1, \dots, t_c) = \mathbf{t}\in \{0, 1, \dots, p-1\}^c$ to be the least vector such that 
  $f^{\mathbf{t}}g^p \in \mathfrak{m}^{[q]}$, in the sense that if $\mathbf{s} = (s_1, \dots, s_c)$ with $s_j \le t_j$ for all $j$ and 
  $s_\ell < t_\ell$ for at least one $\ell$, then $f^\mathbf{s}g^p \not\in \mathfrak{m}^{[q]}$.
  Note that $\mathbf{t}$ exists since $f^{p-1}g^p \in \mathfrak{m}^{[q]}$, and furthermore at least one $t_j \ne 0$ since $g \not\in \mathfrak{m}^{[q/p]}$ 
  (as $[g/x^{q/p}] \ne 0)$ and thus $g^p \not\in \mathfrak{m}^{[q]}$. Without loss of generality, we assume $t_1 \ne 0$. For $1 \le j \le c$ write 
  \[ 
    \hat{f_j} = \prod_{i \in \mathcal{C}} f_i, \;\text{where $\mathcal{C} = \{ i \in \{1, \dots, c\} \,|\, i \ne j \text{ and } t_i > 0\}$.} 
  \]
  Let $\mathbf{t'} = (t_1 -1, t_2, \dots, t_c)$; we know $f^\mathbf{t'}g^p \not\in \mathfrak{m}^{[q]}$ by definition of $\mathbf{t}$. 

  \begin{claim}
    $f^\mathbf{t'}g^p \in (\mathfrak{m}^{[q]}: \mathrm{Jac}(R))$. 
  \end{claim}
  \begin{pfClaim}
    Define $\mathbf{t}^*$ by $\mathbf{t}^*_j = \max\{t_j - 1, 0\}$ for $1 \le j \le c$. If $\partial_i$ is the partial derivative with respect to 
    $x_i$, then $\partial_i(\mathfrak{m}^{[q]}) \subseteq \mathfrak{m}^{[q]}$ for all $i$ and so 
    \begin{align*}
      \partial_i(f^{\mathbf{t}}g^p) &= g^pf^{\mathbf{t}^*}\left(\sum_{j = 1}^c t_j\hat{f_j} \partial_i(f_j)\right) \equiv 0 \pmod{\mathfrak{m}^{[q]}}. 
    \end{align*}
    For any choice of $\mathbf{i} = (i_1, \dots, i_c)$ with $0 \le i_1 < i_2 < \cdots < i_c \le n$, we have a matrix 
    equation
    \[ 
      f^{\mathbf{t}^*} g^p
      \begin{pmatrix}
        \partial_{i_1}(f_1) & \cdots & \partial_{i_1}(f_c) \\
        \vdots & \ddots & \vdots \\
        \partial_{i_c}(f_1) & \cdots & \partial_{i_c}(f_c)\\
      \end{pmatrix} 
      \begin{pmatrix}
        t_1 \hat{f_1} \\ \vdots \\ t_c \hat{f_c}
      \end{pmatrix}
      \equiv 
      \begin{pmatrix}
        0 \\ \vdots \\ 0
      \end{pmatrix}
      \pmod{\mathfrak{m}^{[q]}}. 
    \]
    Calling the determinant of the above matrix of partial derivatives $\Delta_\mathbf{i}$, after multiplying by the adjugate\footnote{The 
    {\em adjugate} $\mathrm{adj}(B)$ of an ($m\times m$) matrix $B$ has the property $\mathrm{adj}(B)B = \det(B)I_m$.} above we have
    \[ 
      t_j \hat{f_j} f^{\mathbf{t}^*}g^p \Delta_\mathbf{i} \equiv 0 \pmod{\mathfrak{m}^{[q]}} \text{ for all $1 \le j \le c$. }
    \]
    Since $\mathrm{Jac}(R)$ is generated by the determinants $\Delta_\mathbf{i}$ ranging over all choices of $\mathbf{i}$ and $0 < t_1 < p$ we conclude
    \[ 
      f^\mathbf{t'}g^p\mathrm{Jac}(R) = \hat{f_1} f^{\mathbf{t}^*}g^p \mathrm{Jac}(R) \subseteq \mathfrak{m}^{[q]}. 
    \]
  \end{pfClaim}

  We now claim $f^\mathbf{t'}g^p \in (\mathfrak{m}^{[q]} : (f_1, f_2^{p-t_2}, \dots, f_c^{p - t_c}))$. Indeed, we defined $\mathbf{t}$ to have the property 
  $f_1f^\mathbf{t'}g^p = f^\mathbf{t}g^p \in \mathfrak{m}^{[q]}$; if $j > 1$ then  $f_j^pg^p$ divides $f_j^{p-t_j}f^\mathbf{t'}g^p$ and we know 
  $f_j^p g^p \in \mathfrak{m}^{[q]}$. 

  Since $\mathrm{Jac}(R) + J$ is $\mathfrak{m}$-primary, there must exist a sequence of determinants $\underline{\Delta} := \Delta_1, \dots, \Delta_{n+1-c}$ so
  that $\underline{\Delta}, \underline{f}$ gives a maximal regular sequence on $S$. Then $\underline{\Delta}, f_1, f_2^{p-t_2}, \dots, f_c^{p-t_c}$ is also 
  a regular sequence, and setting
  $\mathfrak{b} = (\underline{\Delta}, f_1, f_2^{p-t_2}, \dots, f_c^{p - t_c})$ we have $\mathfrak{b} \subseteq \mathrm{Jac}(R) + (f_1, f_2^{p-t_2}, \dots, f_c^{p-t_c})$.  If 
  we set
  \[ \ell = (n+1-c)(d-c) + d_1 + \left(\sum_{j=2}^c (p - t_j) d_j\right) - n \]
  then a Hilbert series argument similar to the one found in the proof of \eqref{Degree Bound} shows that $\mathfrak{m}^\ell \subseteq \mathfrak{b}$. We will use
  this containment to bound the degree of $f^\mathbf{t'}g^p$. Towards this end, we know
  \begin{align*} 
    f^\mathbf{t'} g^p &\in (\mathfrak{m}^{[q]}: \mathrm{Jac}(R) + (f_1, f_2^{p-t_2}, \dots, f_c^{p-t_c})) \\
                              &\subseteq (\mathfrak{m}^{[q]}: \mathfrak{b})\\
                              &\subseteq (\mathfrak{m}^{[q]}: \mathfrak{m}^\ell)\\
                              &= \mathfrak{m}^{[q]} + \mathfrak{m}^{q(n+1) - n - \ell}. 
  \end{align*}
  Since $f^\mathbf{t'} g^p \not\in \mathfrak{m}^{[q]}$ we conclude $f^\mathbf{t'}g^p \in \mathfrak{m}^{q(n+1)- n - \ell}$. This implies
  \begin{align*}
    q(n+1) - n - \ell &= q(n+1) - n - (n + 1 - c)(d - c) - d_1 - \left(\sum_{j=2}^c (p - t_j)d_j\right) + n\\
                      &= q(n+1) - (n + 1 - c)(d - c) + (p - 1)d_1 - pd + \left(\sum_{j=2}^c t_jd_j\right) \\
                      &\le \deg(f^\mathbf{t'}g^p)\\
                      &= (t_1 - 1)d_1 + \left(\sum_{j=2}^c t_jd_j\right) + p\deg(g). 
  \end{align*}
  Now recalling that $p\deg(g) \le q(n+1) - pd - p$ and $t_1 - 1 < p - 1$, we have
  \[
    q(n+1) + (p - 1)d_1 - pd - (n + 1 - c)(d - c) < q(n+1) + (p - 1)d_1 - pd - p
  \]
  which simplifies to  $p < (n+1-c)(d-c)$.
\end{proof}

\begin{exa}
  Let $f = x^2y^2 + y^2z^2 + z^2x^2 \in k[x,y,z] = S$ with $\mathrm{char}(k) > 2$. Then $\tau = \mathfrak{m}$ but $f$ does not have an isolated 
  singularity. In this case, the Bhatt-Singh result \cite[Theorem 3.5]{Bhatt-Singh} does not apply. Theorem \ref{Injectivity Theorem} now tells 
  us that the Frobenius action on $H^2_\mathfrak{m}(S/fS)$ is injective in degrees $\le 0$. Note that in this case, 
  $H^2_\mathfrak{m}(S/fS)_1 \ne 0$ but $H^2_\mathfrak{m}(S/fS)_{\ge 2} = 0$ so the Frobenius action on $H^2_\mathfrak{m}(S/fS)_1$ is zero.
\end{exa}

\begin{exa}
  A smooth projective variety $X = \mathrm{Proj}(R)$ is a {\em Calabi-Yau variety} if $\omega_X = \mathscr{O}_X$. If $R$ is a complete 
  intersection, this is equivalent to $a(R) = 0$. In this case, we can combine theorems \ref{Injectivity Theorem} and \ref{Bhatt-Singh style} 
  to conclude that for $p \gg 0$ the Frobenius action on local cohomology of $R$ can only fail to be injective in degree $0$. 
  Thus, $\tau \in \{\mathfrak{m}, S\}$ for Calabi-Yau complete intersections and all $p \gg 0$. 
\end{exa}


\begin{thebibliography}{4}
  \bibitem{BMS}
    Manuel Blickle, Mircea Mustata, and Karen E. Smith.
    Discreteness and rationality of $F$-thresholds.
    {\em The Michigan Math. J.},
    57:43-61,
    08 2008.

  \bibitem{Bhatt-Singh}
    Bhargav Bhatt and Anurag K. Singh,
    The $F$-pure threshold of a Calabi-Yau hypersurface.
    {\em Mathematische Annalen},
    362(1):551-567, 2014.

  \bibitem{Fedder}
    Richard Fedder,
    $F$-purity and rational singularity.
    {\em Trans. Amer. Math. Soc.},
    278(2):461-480,
    1983.

  \bibitem{FedderGradedCI}
    Richard Fedder,
    $F$-purity and rational singularity in graded complete intersection rings.
    {\em Trans. Amer. Math. Soc.},
    301(1):47-62, 
    1987.
\end{thebibliography}
\end{document}